\documentclass{article}
\usepackage{amssymb,amsmath,multirow}
\usepackage[dvips]{graphicx}
\usepackage{epic,eepic}
\usepackage{chemarrow}
\usepackage{tikz}
\usepackage{color}
\newtheorem{thm}{Theorem}
\newtheorem{lem}{Lemma}
\newtheorem{cor}{Corollary}
\newtheorem{exam}{Example}
\def\pn{\par\smallskip\noindent}
\def\proof{\pn {Proof.} }
\def\endproof{\hfill \quad{$\Box$}\smallskip}

\title{Solving the Tikhonov Regularization of the Total Least Squares Problem by matrix simultaneous diagonalization via congruence }

\author{Van-Bong Nguyen\footnotemark[1] , Thi Ngan Nguyen\footnotemark[2] }

\begin{document}
\maketitle
\renewcommand{\thefootnote}{\fnsymbol{footnote}}
\footnotetext[1]{Department of Mathematics, Tay Nguyen
	University, Vietnam({\tt nvbong@ttn.edu.vn})}
\footnotetext[2]{Department of Mathematics, Tay Nguyen
	University, Vietnam({\tt nguyenthingan@ttn.edu.vn})}

\begin{abstract}
We are concerned with solving the Tikhonov regularization of the total least squares (TRTLS)
 which minimizes a sum of a quadratic fractional function and a nonnegative penalty quadratic form. 
  If the penalty is nonsingular, the problem is attained  and has been solved efficiently. However, 
  if the penalty is singular, the problem becomes much more complicated, it may not be attained. 
 In this paper we explore the property of simultaneously diagonalizable via congruence
  of matrices to make a change of variables and transform the problem into one of three tractable formats. Each of which is then 
  thorough analyzed.

\end{abstract}


\pagestyle{myheadings} \thispagestyle{plain}
\section{Introduction}\label{sec:intro}
When considering  an over-estimated system of linear equations $ Ax\approx b,$
where the coefficient matrix $A\in\Bbb R^{m\times n}$ and the observed vector  $b\in \Bbb R^m$ are both contaminated by noises,
the total least squares (TLS) method \cite{Golub80,Van,Golub96} is applied
to seek a perturbation  $\Delta A \in\Bbb R^{m\times n}$ to the coefficient matrix $A$  and a perturbation  $\Delta b \in\Bbb R^m$
 to the observed vector $b$ such that the total squares $\| \Delta A\|^2_F+\|\Delta b\|^2$ is minimized over the constraints $(A+\Delta A)x=b+\Delta b$ of variables
 $\Delta A, \Delta b, x:$
 \begin{align}\label{TLS}
   \min_{\Delta A, \Delta b,x}~~\Big\{  \| \Delta A\|^2_F+\|\Delta b\|^2 \Big\}~~ \text{   subject to } (A+\Delta A)x=b+\Delta b,
 \end{align}
 where $\|\cdot\|_F$ denotes the Frobenius norm of a matrix and $\|\cdot \|$ is the usual Euclidean norm.  Since the original system  
   is  ill-conditioned, \eqref{TLS} may yield physically meaningless solutions \cite{Tikhonov0,Tikhonov1}. The regularization is thus necessary to stabilize the computed solutions.
 Among many others, the regularized least squares \cite{Gander,Bjorck,Beck06} and the Tikhonov regularization \cite{Tikhonov0,Bjorck,Beck06} are well-established.
 The former approach is to add
 a quadratic constraint to bound the size of the solutions. It leads to solving a quadratic problem with one quadratic
 constraint also referred to  as the trust region subproblem (TRS) \cite{Sorensen,More-Sorensen83,More93,Gould}. The latter one
  adds a quadratic penalty to the objective function  \cite{Gene}:
\begin{align*}
 {\rm (TRTLS)}\hspace*{0.1cm}  \text{min }\Big\{\| \Delta A\|^2_F+\|\Delta b\|^2+\rho\|Lx\|^2\Big\}\text{ subject to\ } (A+\Delta A)x=b+\Delta b,
\end{align*}
where  $L\in\Bbb R^{k\times n}, k\le n,$ is a full row rank Tikhonov regularized matrix and $\rho>0$ is the penalty parameter
chosen to control the size of the solution vector. As indicated in \cite{Bong1}, the Tikhonov regularization problem is a special case of the generalized $p$-regularized subproblems. 

For a fixed $x,$ (TRTLS) is a linearly constrained convex quadratic problem of variables $\Delta A, \Delta b.$ Then the
 necessary  and sufficient optimality condition helps to reformulate it into the following form \eqref{Pro}, please see  \cite{Beck06} 
for more detailed:
\begin{align}\label{Pro}
{\rm  (TRTLS)}  \hspace*{0.7cm}  \min_{x\in\Bbb R^n}\left\{ \dfrac{\|Ax-b\|^2}{\|x\|^2+1}+\rho\|Lx\|^2\right\}.
\end{align}
We emphasize that solving \eqref{Pro}  is still  very difficult since it is nonconvex. 
So far, only the case $k=n,$ i.e., $L$ is a square nonsingular matrix, can be solved efficiently due to its attractive properties. 
Indeed, in that case, the matrix $L^TL$ of the penalty $\|Lx\|^2=x^TL^TLx$ is positive definite: $L^TL\succ0$
so that (TRTLS) can be transformed to the Tikhonov identical (TI) regularized total least squares \cite{Yang}. i.e., the penalty is   
$\rho\|x\|^2.$
Then, by setting $\alpha=\|x\|^2+1,$  (TRTLS)  becomes a minimization problem of 
 a one-dimensional parametric function
\begin{align}\label{original1}
 \hspace*{0.6cm}
\begin{array}{lll}
\mathcal{G}(\alpha):= \min_{\|x\|^2=\alpha-1} \left\{\dfrac{\|Ax-b\|^2}{\alpha}+\rho\|x\|^2\right\}.
\end{array}
\end{align}
 Problem 
 \eqref{original1} is then solved very efficiently by the TRTLSG  Algorithm \cite{Beck06} since 
 $\mathcal{G}(\alpha)$  is strictly unimodal.  On the other hand, 
  due to its special form, (TI) can also be recast as  a single ratio fractional problem
\begin{align}\label{SingleRatio}
  \min_{x\in\Bbb R^n}\dfrac{f(x)}{g(x)},
\end{align}
where $f(x)=\|Ax-b\|^2+\rho\|x\|^4+\rho\|x\|^2$ is a four degree polynomial and $g(x)=\|x\|^2+1$  \cite{Beck06,Yang}.
By the classical Dinkelbach’s method \cite{Dinkelbach}, \eqref{SingleRatio} is equivalent to finding
a  zero point of a univariate function $\pi(t):$
$$\pi(t):=\min_{x\in\Bbb R^n}\{f(x)-t g(x)\}.$$
More importantly, $\pi(t)$ is a strictly decreasing function and its zero point $t^*$ is in the closed interval
$[0,\|b\|^2]$    so that  $t^*$ can be found efficiently by  the  TRTLSI Algorithm \cite{Beck06}.
In a more recent study, Yang et al. \cite{Yang} explore the hidden convexity of the (TI) 
to reformulate it into a  problem of finding a zero point of a strictly decreasing, smooth and convex univariate function over the interval $(-\lambda_{\min}(A^TA),+\infty)$ 
and propose a generalized Newton’s method referred to as the GN algorithm for finding such a unique zero point. 
 
 However, if  $k<n,$ (TRTLS) becomes much more complicated. 
 Only study whether it is attained is already very difficult. So far, only a sufficient optimality condition can be obtained, please see 
 condition (13) in \cite{Beck06}.  Solving (TRTLS) in this case is thus still open and it inspires us in the present study.
 For convenience in presenting, we
 consider (TRTLS) in the following format
 
\begin{align*}
{\rm(P)} \hspace*{0.7cm}\lambda^*:= \min_{x\in\Bbb R^n}\left\{\mathcal{H}(x)=\dfrac{x^TWx-2a^Tx+\alpha}{\|x\|^2+1}+x^TBx\right\},
\end{align*}
where  $W, B\in S^n_+$  are positive semidefinite symmetric matrices, and in particular, $B$ is singular, $a\in\Bbb R^n,$ 
$\alpha\in\Bbb R$  such that $g(x):=x^TWx-2a^Tx+\alpha\ge0 ~\forall x\in\Bbb R^n.$
  We note that  $g(x)\ge0 ~\forall x\in\Bbb R^n$ implying  $a\in{\rm range}(W),$
i.e., $a=Wx_0$ for some $x_0\in\Bbb R^n.$ If $W=A^TA,$  $B=\rho L^TL$ with $\rho>0,$ 
$b=Ax_0$ and  $\alpha=\|b\|^2$ then  (P) 
is reduced to (TRTLS). 

 We will show that there is always a nonsingular matrix $P$ that puts 
  $B$ into the following form
 \begin{align*} 
&  P^TBP={\rm diag}(\alpha_1, \alpha_2, \ldots, \alpha_k,\underbrace{0,\ldots,0}_{n-k\ \text{  zeros }}),
\end{align*}
with $\alpha_i>0$ and simultaneously puts $W$ into one of the following three cases:
\begin{align}
   &  P^TWP={\rm diag}(\beta_1, \beta_2, \ldots, \beta_k,\underbrace{0,\ldots,0}_{n-k\ \text{  zeros }}),\label{In1}\\
  & P^TWP={\rm diag}(\beta_1, \beta_2, \ldots, \beta_k,\gamma_1,\ldots,\gamma_s,\underbrace{0,\ldots,0}_{n-k-s\ \text{  zeros }}),\label{In2}\\
  & P^TWP={\rm diag}(\beta_1, \beta_2, \ldots, \beta_k,\underbrace{\gamma_1,\ldots,\gamma_s}_{s=n-k})\label{In3}
\end{align}
 with $\beta_i\ge0$ and $\gamma_i>0.$ 
 These allow to make a change of variables $x=Py$ to convert (P) to one of three simpler formats, please see
 \eqref{dang1}, \eqref{dang3} and \eqref{dang2} below,
 corresponding to one of the three cases of $P^TBP$ and $P^TWP$ above.
 As we will see, solving \eqref{dang1} and \eqref{dang3} is easy and we can define clearly their optimal solution sets. 
 However, solving \eqref{dang2} is harder. We then need to decompose the matrix pencil into 
 blocks that can help to characterize the attainment of \eqref{dang2} and evaluate its optimal value. 
\section{A simplified reformulation of  (P) by SDC}\label{sec02}
Since $B\succeq0$ and $W\succeq0,$ $B$ and $W$ are SDC \cite[Lemma 5]{Bong}, i.e., there exists a nonsingular matrix $P$ such that $P^TBP$ and $P^TWP$ are diagonal. This suggests to make a change of variables to simplify the formulation of (P). We first propose the following 
important result.

\begin{lem}\label{B1}   With $B$ and $W$ given in (P), there always exists a nonsingular matrix $P$
that puts  $B$ to the following format:
\begin{align}\label{a0}
\bar{B}:= P^TBP={\rm diag}(\alpha_1, \alpha_2, \ldots, \alpha_k,0,\ldots,0),
\end{align}
where $\alpha_i>0, i=1,2,\ldots,k,$
and simultaneously puts $W$   in one of the following three forms:
\begin{align}
\bar{W}:= & P^TWP={\rm diag}(\beta_1, \beta_2, \ldots, \beta_k,0,\ldots,0),\label{a}\\
 \bar{W}:=&  P^TWP={\rm diag}(\beta_1, \beta_2, \ldots, \beta_k,\gamma_1,\ldots,\gamma_s,0,\ldots,0),\label{c}\\
\bar{W}:= &  P^TWP={\rm diag}(\beta_1, \beta_2, \ldots, \beta_k,\gamma_1,\ldots,\gamma_s)\label{b},
\end{align}
such that  $\beta_i\ge0, i=1,2,\ldots,k;$ and $\gamma_i>0, i=1,2,\ldots, s;$ here $k+s<n$ in \eqref{c} but $k+s=n$ in \eqref{b}.
\end{lem}
\proof 
Since $B$ is singular, there always exists a nonsingular matrix $U$  such that
\begin{align}\label{1}
 U^TBU=\left(\begin{matrix}B_1&0_{k\times(n-k)}\\0_{(n-k)\times k}&0_{(n-k)\times (n-k)} \end{matrix}\right)
\end{align}
and $U^TWU$ takes one of the following forms:
\begin{align}\label{2}
 U^TWU=\left(\begin{matrix}W_1&W_2\\ W_2^T&0_{(n-k)\times(n-k)} \end{matrix}\right), 
\end{align}
\begin{align}\label{3}
 U^TWU=\left(\begin{matrix}W_1&0_{k\times s}&W_2\\0_{s\times k}&W_3&0_{s\times(n-k-s)}\\ W_2^T&0_{(n-k-s)\times s}&0_{(n-k-s)\times(n-k-s)} \end{matrix}\right),
\end{align}
where $B_1$ is a $k\times k$ nonsingular diagonal matrix, $W_1$ is a $k\times k$ symmetric matrix, and  $W_2$
is a $k\times(n-k)$ matrix in \eqref{2} but in \eqref{3} it 
 either is a  $k\times(n-k-s)$ matrix or does not exist, $W_3$ is an
$s\times s$ nonsingular diagonal matrix. Please see   \cite[Lemma 4]{Ngan} or \cite[Lemma 3]{Bong} for more detailed.
Here $0_{p \times q}$ denotes the $p\times q$ zero matrix.  
 
\begin{enumerate}
  \item If $B$ and $W$ are decomposed into \eqref{1} and \eqref{2}, respectively, then since $B$ and $W$ are SDC, $W_2$ must be zero: $W_2=0_{k\times(n-k)}$ 
  and $B_1, W_1$ are SDC, please see  \cite[Lemma 11]{Ngan}. So
  $U^TBU$ and $U^TWU$ are of the forms:
  $$U^TBU={\rm diag}(B_1,0_{(n-k)\times(n-k)}),~  U^TWU={\rm diag }(W_1, 0_{(n-k)\times(n-k)})$$
   and there is a nonsingular matrix $V$  such that
 $$V^TB_1V={\rm diag}(\alpha_1, \alpha_2, \ldots, \alpha_k), V^TW_1V={\rm diag}(\beta_1, \beta_2, \ldots, \beta_k).$$
 Let $Q={\rm diag}(V, I_{n-k})$ and $P=UQ,$ we then have
$$
   P^TBP=Q^TU^TBUQ={\rm diag}(\alpha_1, \alpha_2, \ldots, \alpha_k,0,\ldots,0),$$
 $$  P^TWP=Q^TU^TWUQ={\rm diag}(\beta_1, \beta_2, \ldots, \beta_k,0,\ldots,0)$$
  as desired \eqref{a0} and \eqref{a} with  $\alpha_i>0$  and
 $\beta_i\ge0$ for $ i=1,2,\ldots,k.$
  \item If $B$ and $W$ are decomposed into
\eqref{1} and \eqref{3}, respectively, then, again since $B$ and $W$ are SDC,  $W_2$ either is zero or does not exist and $B_1, W_1$ are SDC, please see \cite[Theorem 6]{Jiang}
or \cite[Lemma 12]{Ngan}. 
 Let $V$ be a nonsingular matrix such that
$$V^TB_1V={\rm diag}(\alpha_1, \alpha_2, \ldots, \alpha_k), V^TW_1V={\rm diag}(\beta_1, \beta_2, \ldots, \beta_k)$$
and set $Q={\rm diag}(V,  I_{n-k}), P=UQ.$ We then have the following subcases.
\begin{itemize}
  \item If $W_2=0,$ then 
  \begin{align*} 
&U^TBU={\rm diag}(B_1, 0_{s\times s}, 0_{(n-k-s)\times(n-k-s)}), \\
&U^TWU={\rm diag}(W_1, W_3,0_{(n-k-s)\times(n-k-s)} )
\end{align*}
  and
\begin{align*}
  &\  P^TBP=Q^TU^TBUQ={\rm diag}(\alpha_1, \alpha_2, \ldots, \alpha_k,0,\ldots,0, 0, \ldots,0), \\
   &  P^TWP=Q^TU^TWUQ={\rm diag}(\beta_1, \beta_2, \ldots, \beta_k,\gamma_1,\ldots,\gamma_s,0,\ldots,0).
\end{align*} 
 These prove \eqref{a0} and \eqref{c}.
  \item  If  $W_2$  does not exist then
\begin{align}\label{5}
U^TBU={\rm diag}(B_1, 0_{(n-k)\times (n-k)}), ~ U^TWU={\rm diag}(W_1, W_3)
\end{align}
and 
\begin{align*}
  &\  P^TBP=Q^TU^TBUQ={\rm diag}(\alpha_1, \alpha_2, \ldots, \alpha_k,0,\ldots,0), \\
   &  P^TWP=Q^TU^TWUQ={\rm diag}(\beta_1, \beta_2, \ldots, \beta_k,\gamma_1,\ldots,\gamma_s).
\end{align*} 
These are
 \eqref{a0} and \eqref{b} as desired.
\end{itemize}
\end{enumerate}

\endproof

Lemma \ref{B1} suggests a change of variables $x=Py$ such that
\begin{align*}
  x^TWx-2a^Tx+\alpha&=y^T(P^TWP)y-2(P^Ta)^Ty+\alpha=y^T\bar{W}y-2\bar{a}^Ty+\alpha, \\
    x^TBx&=y^T(P^TBP)y=y^T\bar{B}y, \\
    \|x\|^2&=x^Tx=y^T(P^TP)y=y^TMy,
\end{align*} 
  here $\bar{a}=P^Ta,$ $M=P^TP\succ0.$ Now, depending on each case of
$\bar{B}$ and $\bar{W},$ (P) is  reformulated as follows.
\begin{enumerate}
  \item If $\bar{B}$ and $\bar{W}$ take the forms \eqref{a0} and \eqref{a}, respectively, then we set  $A_1={\rm diag}(\beta_1,\beta_2,\ldots,\beta_k), D_1={\rm diag}(\alpha_1,\alpha_2,\ldots,\alpha_k)$ and separate $y$ into two parts:
  $$y=\left(\begin{matrix}u\\ v\end{matrix}\right):=u\oplus v,$$
  with $u\in\Bbb R^k, v\in\Bbb R^{n-k}$ such that  
  $$g(Py)=y^T\bar{W}y+2\bar{a}^Ty+\alpha= u^TA_1u-2p^Tu-2q^Tv+\alpha,$$ 
   where $\bar{a}= p\oplus q,$   $p\in\Bbb R^{k}, q\in\Bbb R^{n-k}.$
   Since $g(x)=g(Py)\ge0~\forall y\in\Bbb R^n,$ we have  
   $\bar{a}\in {\rm range}(\bar{W})$ and thus $q=0.$  Then
(P)  can be written in the following concise form, where the numerator of the quadratic fraction contains no $v:$
  \begin{align}\label{dang1}
\lambda^*=\min_{u\oplus v \in\Bbb R^{k}\times \Bbb R^{n-k}}
\left\{\mathcal{H}(u,v):=\dfrac{u^TA_1u-2p^Tu+\alpha}{(u\oplus v)^TM(u\oplus v)+1}+u^TD_1u\right\}.
\end{align}
\item If $\bar{B}$ and $\bar{W}$ take the forms \eqref{a0} and \eqref{c}, respectively, then we set $A_1, D_1$ as above, 
  $W_3={\rm diag}(\gamma_1,\gamma_2,\ldots,\gamma_s)$ and separate
$y=u\oplus  v\oplus \omega,$ where $u\in\Bbb R^k, v\in\Bbb R^s, \omega\in\Bbb R^{n-k-s}.$ Let
 $\bar{a}=p\oplus q\oplus d,$  
  with  $p\in\Bbb R^k, q\in\Bbb R^s, d\in\Bbb R^{n-k-s}$ and write $g(x)$ as follows
  $$g(x)=g(Py)=y^T\bar{W}y+2\bar{a}^Ty+\alpha= u^TA_1u-2p^Tu+v^TW_3v-2q^Tv-2d^T\omega+\alpha.$$ 
   Again,  $g(x)\ge0~\forall x\in\Bbb R^n$ implies $d=0.$ Then  (P) takes the following form:
\begin{align}\label{dang3}
 \lambda^*= \min\left\{\mathcal{H}(u,v,\omega)=
\dfrac{u^TA_1u-2p^Tu+v^TW_3v-2q^Tv+\alpha}{(u\oplus v\oplus\omega)^TM(u\oplus v\oplus\omega)+1}+u^TD_1u\right\},
\end{align}
with ${u\oplus v\oplus\omega \in\Bbb R^{k}\times \Bbb R^s\times\Bbb R^{n-k-s}}.$
\item If $B$ and $W$ are put to \eqref{a0} and \eqref{b}, respectively, we then set $y= u\oplus v,$
 $\bar{a}=p\oplus q,$ where $u,p\in\Bbb R^k, v, q\in\Bbb R^{n-k}$ and rewrite (P) as follows:
  \begin{align}\label{dang2}
 \lambda^*= \min\left\{\mathcal{H}(u,v)=\dfrac{u^TA_1u-2p^Tu+v^TW_3v-2q^Tv+\alpha}{(u\oplus v)^TM(u\oplus v)+1}+u^TD_1u\right\},
\end{align}
with $ u\oplus v \in\Bbb R^{k}\times \Bbb R^{n-k}.$ In this case, $s=n-k.$  
\end{enumerate}

 \section{Solving (P)}
 
 We are now ready to solve (P) which has been reformulated into one of three cases \eqref{dang1}, \eqref{dang3} and \eqref{dang2}.  
 Firstly, \eqref{dang1} is completely characterized as follows.
\begin{thm}\label{Thm1}
Problem \eqref{dang1} is attained if and only if  $\alpha=0.$ If it is attained, its optimal solutions are then $\{x=0\oplus v: ~v\in\Bbb R^{n-k}\}.$
\end{thm}
\proof
Observe simply that if $\alpha=0$ then (P) is attained at the optimal solutions
$x= 0\oplus v$ for arbitrary $v\in\Bbb R^{n-k}$ with  optimal value zero.
Now suppose (P) is attained, we need  to show that $\alpha=0.$
Let us consider the following two problems:
   \begin{align}\label{Pro1}
	\begin{array}{llll}
		\gamma^*=&\min & u^TD_1u& \\
		&{\rm s.t.}& h(u):= u^TA_1u+2p^Tu+\alpha=0
	\end{array}
\end{align}
and
\begin{align}\label{Pro2}
\eta^*=\min_{u\in\Bbb R^k} u^TD_1u.
\end{align}
We have $\eta^*=0,$ and $\gamma^*\ge \eta^*.$
Suppose  $y^*= u^*\oplus v^* $ is an optimal solution of \eqref{dang1}. Then it must hold that $h(u^*)=0.$
 Indeed, if  $h(u^*)=\beta^*>0$ then
 \begin{align}\label{ct1}
   \lambda^*=&\mathcal{H}(u^*,v^*)=\dfrac{{u^*}^TA_1u^*-2p^Tu^*+\alpha}{(u^*\oplus v^*)^TM(u^*\oplus v^*)+1}+{u^*}^TD_1u^*\\
   =& \dfrac{\beta^*}{(u^*\oplus v^*)^TM(u^*\oplus v^*)+1}+{u^*}^TD_1u^*.\nonumber
 \end{align}
We write $M=\left(\begin{matrix}
                    M_1 & M_2 \\
                    M_2^T & M_3
                  \end{matrix}\right)$
such that $M_1$ is a $k\times k$ matrix, $M_2$ is a $k\times (n-k)$ matrix and $M_3$ is a $(n-k)\times(n-k)$ matrix.
Since $M\succ0,$ we have $M_3\succ0.$
Then
\begin{align*}
  \lambda^*=&\dfrac{\beta^*}{(u^*\oplus v^*)^TM(u^*\oplus v^*)+1}+{u^*}^TD_1u^*\\
  =&\dfrac{\beta^*}{{v^*}^TM_3v^*+2(M_2^Tu^*)^Tv^*+{u^*}^TM_1u^*+1}+{u^*}^TD_1u^*\\
  \le&\dfrac{\beta^*}{v^TM_3v+2(M_2^Tu^*)^Tv+{u^*}^TM_1u^*+1}+{u^*}^TD_1u^* \underbrace{  \longrightarrow }_{\text{ as } \|v\|\rightarrow\infty}{u^*}^TD_1u^*<\lambda^*.
\end{align*}
This contradiction shows that $h(u^*)=\beta^*=0.$ It also indicates that $u^*$ must be feasible to \eqref{Pro1} and thus $ {u^*}^TD_1u^*\ge\gamma^*.$
If  ${u^*}^TD_1u^*>\gamma^*$ then  there must be a feasible solution $\bar{u}$ of \eqref{Pro1} such that
$\bar{u}^TD_1\bar{u}<{u^*}^TD_1u^*.$ Then $\mathcal{H}(\bar{u},v^*)=\bar{u}^TD_1\bar{u}<{u^*}^TD_1u^*=\mathcal{H}(u^*,v^*).$ 
This is again a contradiction. So it must hold that ${u^*}^TD_1u^*=\gamma^*.$ As a result, we have
$\lambda^*={u^*}^TD_1u^*=\gamma^*.$ Finally, we need to show that $\gamma^*=\eta^*=0.$
Indeed, if  $\gamma^*>\eta^*$ then there must exists an optimal solution $u'$ of \eqref{Pro2} such that $u'$ is not feasible to
\eqref{Pro1}. That is, $h(u')={u'}^TA_1u'-2p^Tu'+\alpha>0$  and
$\gamma^*>{u'}^TD_1u'.$ Then, since $M_3\succ0,$ we have
$$\mathcal{H}(u',v)=\dfrac{{u'}^TA_1u'-2p^Tu'+\alpha}{v^TM_3v+2(M_2^Tu')^Tv+{u'}^TM_1u'+1}+{u'}^TD_1u'\longrightarrow {u'}^TD_1u'<\gamma^*=\lambda^*$$
as $\|v\|\rightarrow \infty.$  This is again a contradiction and it shows that  $\gamma^*=\eta^*.$ Since $\eta^*=0,$  
we have $\gamma^*=\lambda^*=0$ and the optimal solutions then must be $(0,v)$
and $\alpha=0.$
\endproof

\begin{thm}\label{Thm2}
Problem \eqref{dang3} is attained, say at $u^*\oplus v^*\oplus \omega^*,$ if and only if  $u^*=0,$  ${v^*}^TW_3v^*-2q^Tv^* +\alpha=0$
and $\omega^*\in\Bbb R^{n-k-s}$ is arbitrary.

\end{thm}
\proof
Observe first that if $u^*=0$ and ${v^*}^TW_3v^*-2q^Tv^* +\alpha=0$
then $u^*\oplus v^*\oplus \omega$ is optimal to (P) for all $\omega \in\Bbb R^{n-k-s}$ and the optimal value of (P) is then zero.
For the converse, suppose  \eqref{dang3} is attained at $u^*\oplus v^*\oplus \omega^*.$ 
Let $h(u,v)= u^TA_1u-2p^Tu+v^TW_3v-2q^Tv+\alpha$ and we will show that 
$h(u^*,v^*)=0.$
We consider the following two problems
 \begin{align}\label{Pro3}
	\begin{array}{llll}
		\gamma^*=&\min & u^TD_1u& \\
		&{\rm s.t.}& h(u,v)= u^TA_1u-2p^Tu+v^TW_3v-2q^Tv+\alpha=0
	\end{array}
\end{align}
and
\begin{align}\label{Pro4}
\eta^*=\min_{u\in\Bbb R^k} u^TD_1u.
\end{align}
We first have $\gamma^*\ge \eta^*=0.$ Now, suppose on contrary that  $h(u^*, v^*)=\beta^*>0.$  Then
\begin{align*} 
 \lambda^*=  \mathcal{H}(u^*,v^*,\omega^*)=
\dfrac{\beta^*}{(u^*\oplus v^*\oplus\omega^*)^TM(u^*\oplus v^*\oplus\omega^*)+1}+{u^*}^TD_1u^*.
\end{align*}
We divide $M$ into four parts:
$M=\left(\begin{matrix}
                    M_1 & M_2 \\
                    M_2^T & M_3
                  \end{matrix}\right)$
such that $M_1$ is a $(k+s)\times (k+s)$ matrix, $M_2$ is a $(k+s)\times (n-k-s)$ matrix and $M_3$ is a $(n-k-s)\times(n-k-s)$ matrix. 
Since  $M\succ0,$ we have $M_1\succ0$ and $M_3\succ0.$
Then  
\begin{align*}
  \lambda^* &= \dfrac{\beta^*}
  {{\omega^*}^TM_3\omega^*+2 (u^* \oplus v^*)^TM_2 \omega^*+(u^*\oplus v^*)^T M_1(u^* \oplus v^*)+1}+{u^* }^TD_1u^*\\
   & \le \dfrac{\beta^*}
  {\omega^TM_3\omega+2 (u^* \oplus v^*)^TM_2 \omega+(u^*\oplus v^*)^T M_1(u^* \oplus v^*)+1}+{u^* }^TD_1u^* \\
 & \rightarrow {u^* }^TD_1u^*<\lambda^*
\end{align*} 
  as $\|\omega\|\rightarrow\infty.$  This is a contradiction and it requires $h(u^*, v^*)=\beta^*=0.$   
   This also indicates that $u^*\oplus v^*$ is feasible to
  \eqref{Pro3} for any optimal solution  $u^*\oplus v^*\oplus \omega^*$ of (P). That implies
  $\lambda^*=\gamma^*.$ Finally, we will show  $\gamma^*=0.$ We observe that $\gamma^*\ge0.$
  Suppose $\gamma^*>0.$ Then $0\oplus v $ is not feasible to \eqref{Pro3} for any $v\in\Bbb R^s,$ i.e., $h(0,v)>0$
  for all $v\in\Bbb R^s,$ since, if $h(0,v)=0$ for some $v\in\Bbb R^s,$ then $\gamma^*=0.$ Fix a $\bar{v}\in\Bbb R^s$ and let   $\|\omega\|\rightarrow\infty$ we will have
$$\mathcal{H}(0,\bar{v}, \omega)=\dfrac{\bar{v}^TW_3\bar{v}-2q^T\bar{v}+\alpha}{\omega^TM_3\omega+2(0\oplus \bar{v})^TM_2\omega+(0\oplus\bar{v})^T M_1( 0 \oplus \bar{v})+1} \longrightarrow0<\gamma^*=\lambda^*.$$
     This is a contradiction. So it must hold that $\gamma^*=\lambda^*=0,$
   and   ${v^*}^TW_3v^*-2q^Tv^* +\alpha=0.$

\endproof

Finally, we deal with  the  problem \eqref{dang2}. As we will see, solving \eqref{dang2} is more complicated.  
Let us divide $M$ into four parts as follows
 $$M=\left(\begin{matrix}M_1&M_2\\M_2^T&M_3 \end{matrix}\right),$$
 where  $M_1, M_2, M_3$ are the  $k\times k, k\times s$ and $s\times s$   matrices, respectively. 
 The problem \eqref{dang2} is now rewritten as  
 $$\lambda^*=\inf_{u\oplus v\in\Bbb R^k \times\Bbb R^{n-k}} 
 \dfrac{u^TA_1u -2p^Tu +v^TW_3v-2q^Tv+\alpha}{u^TM_1u+2(M_2^Tu)^Tv+v^TM_3v+1}+u^TD_1u.$$
 We have
 $$\mathcal{H}(u,v)=\dfrac{u^TA_1u -2p^Tu +v^TW_3v-2q^Tv+\alpha}{u^TM_1u+2(M_2^Tu)^Tv+v^TM_3v+1}+u^TD_1u\ge u^TD_1u\ge \lambda_{\min}(D_1)\|u\|^2, $$
so $\mathcal{H}(u,v)\rightarrow\infty$ as $\|u\|\rightarrow\infty.$ This indicates that for any sequence 
$(u^j,v^j)\subset \Bbb R^k \times\Bbb R^{n-k}$ such that
 $\mathcal{H}(u^j,v^j)\rightarrow \lambda^*,$ $(u^j)$ must be bounded.  
  To consider the case  $\|v\|\rightarrow \infty$ we denote by $\lambda=\lambda_{\min}(W_3,M_3)$ 
  the smallest generalized eigenvalue of the pair $(W_3,M_3).$ Then 
   for any fixed $u$ and $v\ne0,$ $\dfrac{v}{\|v\|}$ is bounded as $\|v\|\rightarrow\infty,$ there is a subsequence 
   $(v^j)$ such that  $\dfrac{v^j}{\|v^j\|}\rightarrow \omega.$
As a result,
 $$\dfrac{u^TA_1u -2p^Tu +v^TW_3v-2q^Tv+\alpha}{u^TM_1u+2(M_2^Tu)^Tv+v^TM_3v+1} \rightarrow
 \dfrac{ \omega^TW_3\omega}{ \omega^TM_3\omega}\ge \lambda_{\min}(W_3,M_3)=\lambda$$
 as ${\|v\|\rightarrow\infty}.$
 Since $u^TD_1u\ge0,$ we have
 $$\mathcal{H}(u,v)=\dfrac{u^TA_1u -2p^Tu +v^TW_3v-2q^Tv+\alpha}{u^TM_1u+2(M_2^Tu)^Tv+v^TM_3v+1}+u^TD_1u\ge \lambda.$$
Let $u=0, v=tz,$ where $z$ is the generalized eigenvector corresponding to $\lambda_{\min}(W_3,M_3),$
we have $\mathcal{H}(0,tz)\rightarrow \lambda$ as $t\rightarrow\infty.$ This obviously shows that
$$\lambda^*=\inf \mathcal{H}(u,v)\le \lambda=\lambda_{\min}(W_3,M_3).$$
\begin{lem}\label{L1}
  If $\lambda^*<\lambda$ then $\lambda^*$ is  attained. 
\end{lem}
\proof 
Let $(u^j, v^j)$ be a sequence such that $\mathcal{H}(u^j,v^j)\rightarrow \inf \mathcal{H}(u,v)=\lambda^*.$ We will show that
the sequence $(u^j, v^j)$ must be bounded. As seen, $(u^j)$ is bounded. Suppose
$(v^j)$ is unbounded: $\|v^j\|\rightarrow\infty$ as $j\rightarrow\infty.$ Then
\begin{align*}
\mathcal{H}(u^j,v^j)&=\dfrac{\left(\frac{u^j}{\|v\|}\right)^TA_1\frac{u^j}{\|v\|}
 -2p^T\frac{u^j}{\|v\|^2} +\left(\frac{v^j}{\|v\|}\right)^TW_3\frac{v^j}{\|v\|}-2q^T\frac{v^j}{\|v\|^2}+\frac{\alpha}{\|v\|^2}}
 {\left(\frac{u^j}{\|v\|}\right)^TM_1 \frac{u^j}{\|v\|} +2(M_2^T\frac{u^j}{\|v\|})^T\frac{v^j}{\|v\|}+\left(\frac{v^j}{\|v\|}\right)^TM_3\frac{v^j}{\|v\|}+\frac{1}{\|v\|^2}}\\
 & \rightarrow 
 \dfrac{ \omega^TW_3\omega}{\omega^TM_3\omega}  
    \ge \lambda_{\min}(W_3,M_3)=\lambda.
 \end{align*}
 This contradicts to our assumption that $\lambda^*<\lambda.$ So the sequence $(u^j, v^j)$
 must be bounded. Then there is a subsequence $(u^{j_i}, v^{j_i})$ converging to $(u^*, v^*).$ Since 
 $\mathcal{H}(u,v)$ is continuous, it holds that 
 $\mathcal{H}(u^*,v^*)=\lambda^*.$ The problem is thus attained.
 
 \endproof 
 
 However, checking whether $\lambda^*<\lambda$ is another issue.  
 In the following we propose an alternative result which is checkable and characterizes the attainment of the problem. 
 Let $\bar{H}=\left(\begin{matrix}
                          W_3&-q\\
                          -q^T&\alpha
                        \end{matrix}\right), \bar{M}=\left(\begin{matrix}M_3&0\\0&1\end{matrix}\right).$
Then  $W_3$ and $M_3$ are principle submatrices  of $\bar{H}$ and $\bar{M},$ respectively, it always holds that
$\lambda_{\min}(\bar{H},\bar{M})\le\lambda_{\min}(W_3,M_3).$ 
 
 \begin{thm}\label{mainthm1}
  If $\lambda_{\min}(\bar{H},\bar{M})<\lambda_{\min}(W_3,M_3)$ then 
  $\lambda^*\le\lambda_{\min}(\bar{H},\bar{M})$ and it is attained.  
\end{thm}
 \proof
   Let $\bar{v}=(v^T, t)^T\in\Bbb R^{n-k}\times\Bbb R$ 
   be the generalized eigenvector corresponding to $\lambda_{\min}(\bar{H},\bar{M}),$ the smallest generalized eigenvalue of
 the couple $(\bar{H},\bar{M}):$ 
 $$\lambda_{\min}(\bar{H},\bar{M})=\dfrac{\bar{v}^T\bar{H}\bar{v}}{\bar{v}^T\bar{M}\bar{v}}=\dfrac{v^TW_3v
 -2tq^Tv+\alpha t^2}{v^TM_3v+t^2}.$$
   If $t=0$ then  
 $$\lambda_{\min}(\bar{H},\bar{M})=\dfrac{\bar{v}^T\bar{H}\bar{v}}{\bar{v}^T\bar{M}\bar{v}}=\dfrac{v^TW_3v}
 {v^TM_3v}\ge\lambda_{\min}(W_3,M_3).$$ This contradiction shows that $t\ne0.$
 Let $F$ be an $n\times s$ matrix defined as following
  $$F=\left(\begin{matrix}0&0&\ldots&0\\ 0&0&\ldots&0\\
\vdots&\vdots&\ldots&\vdots\\
 0&0&\ldots&0\\
1&0&\ldots&0\\
0&1&\ldots&0\\
\vdots&\vdots&\ldots&\vdots\\
0&0&\ldots&1 \end{matrix}\right)=\left(\begin{matrix}{\bf 0}_{k\times s}\\ I_s  \end{matrix}\right)\in\Bbb R^{n\times s}.$$
Then $F^T\bar{W}F=W_3,  F^TMF=M_3,$  $F^T\bar{B}=0_{s\times n}$ and
$F^T(p\oplus q)=q.$ We have
 \begin{align}
 \lambda_{\min}(\bar{H},\bar{M})=&\dfrac{\bar{v}^T\bar{H}\bar{v}}{\bar{v}^T\bar{M}\bar{v}}=
 \dfrac{\frac{\bar{v}^T}{t}\bar{H}\frac{\bar{v}}{t}}{\frac{\bar{v}^T}{t}\bar{M}\frac{\bar{v}}{t}}= \dfrac{\hat{v}^TW_3\hat{v}-2q^T\hat{v}+\alpha}{\hat{v}^TM_3\hat{v}+1}\nonumber\\
 =& \dfrac{\hat{v}^T F^TWF\hat{v}- 2(F^T(p\oplus q))^T\hat{v}+\alpha}{\hat{v}^TF^TMF\hat{v}+1}+\hat{v}^TF^T\bar{B}F\hat{v}\nonumber\\
 =&\dfrac{(F\hat{v})^T W(F\hat{v})- 2(p\oplus q)^T(F\hat{v})+\alpha}{(F\hat{v})^TM(F\hat{v})+1}+(F\hat{v})^T\bar{B}(F\hat{v})\nonumber\\
 =&\mathcal{H}(F\hat{v})\ge \lambda^* \label{bt1},
 \end{align}
 here  $\hat{v}=\frac{v}{t}\in\Bbb R^{n-k}.$

 Now, suppose  that $\lambda^*$
  is not attained. Then there is a sequence
 $(y_n)=(u_n\oplus v_n), \|y_n\|\rightarrow\infty$ such that
 $$\dfrac{u_n^TA_1u_n-2p^Tu_n+v_u^TW_3v_n-2q^Tv_n+\alpha}{(u_n\oplus v_n)^TM (u_n\oplus v_n)+1}+u_n^TD_1 u_n\rightarrow \lambda^* \text{ as } n\rightarrow\infty. $$
 Since $f(y_n)=f(u_n\oplus v_n):=\dfrac{u_n^TA_1u_n-2p^Tu_n+v_u^TW_3v_n-2q^Tv_n+\alpha}{(u_n\oplus v_n)^TM (u_n\oplus v_n)+1}$ and 
 $\dfrac{y_n}{\|y_n\|}$ are both bounded, there must be a subsequence $(y_{n_k})=(u_{n_k}\oplus v_{n_k})$
  such that 
  $$f(y_{n_k})\rightarrow \eta,~ \dfrac{y_{n_k}}{\|y_{n_k}\|}\rightarrow d  \text{ as }
   k\rightarrow\infty.$$
   We observe that  $ f(y_{n_k})+y_{n_k}^T\bar{B}y_{n_k}  \rightarrow \lambda^*$ as $ k\rightarrow\infty.$
     So
  $$\lim_{k\rightarrow\infty}\dfrac{f(y_{n_k})+y_{n_k}^T\bar{B}y_{n_k}}{\|y_{n_k}\|^2}=0,$$
  where  $\lim_{k\rightarrow\infty}\dfrac{f(y_{n_k})}{\|y_{n_k}\|^2}=0$ since  $f(y_n)$ is bounded. We thus have
    $$\lim_{k\rightarrow\infty}\dfrac{y_{n_k}^T\bar{B}y_{n_k}}{\|y_{n_k}\|^2}=
  \lim_{k\rightarrow\infty}\dfrac{y_{n_k}^T}{\|y_{n_k}\|}\bar{B}\dfrac{y_{n_k}}{\|y_{n_k}\|}=d^T\bar{B}d=0.$$
  Since $\bar{B}\succeq0,$ $d^T\bar{B}d=0$ implies $\bar{B}d=0.$  This shows that $d$ must be $d=Fv$ for
 some $v\in\Bbb R^s.$ Then
\begin{align*}
  \lambda^*= & \lim_{k\rightarrow\infty}\{f(y_{n_k})+y_{n_k}^T\bar{B}y_{n_k} \}\\
  \ge &  \lim_{k\rightarrow\infty}\left\{f(y_{n_k})=\dfrac{y_{n_k}^T\bar{W}y_{n_k}-2(p^T~q^T)y_{n_k}+\alpha}{y_{n_k}^TMy_{n_k}+1}\right\}\\
  =& \lim_{k\rightarrow\infty} \dfrac{\dfrac{y_{n_k}^T}{\|y_{n_k}\|}\bar{W}\dfrac{y_{n_k}}{\|y_{n_k}\|}-2(p^T~q^T)\dfrac{y_{n_k}}{\|y_{n_k}\|^2}+\dfrac{\alpha}{\|y_{n_k}\|^2}}
  {\dfrac{y_{n_k}^T}{\|y_{n_k}\|}M\dfrac{y_{n_k}}{\|y_{n_k}\|} +\dfrac{1}{\|y_{n_k}\|^2}}=\dfrac{d^T\bar{W}d}{d^TMd}=\dfrac{v^TF^T\bar{W}Fv}{v^TF^TMFv}\\
   \ge &\min_{\|Fv\|=1} \dfrac{v^TW_3v}{v^TM_3v} =\min_{\|v\|^2=1} \dfrac{v^TW_3v}{v^TM_3v}=\lambda_{\min}(W_3,M_3) >\lambda_{\min}(\bar{H}, \bar{M}).
 \end{align*}
  This contradicts to \eqref{bt1} obtained above. So $\lambda^*$ must be attained.
 \endproof
 
The last case is $\lambda_{\min}(\bar{H},\bar{M})=\lambda_{\min}(W_3,M_3).$ 
 Let 
 $h_1(u)=u^TA_1u-2p^Tu$ and $h_2(v)=v^TW_3v-2q^Tv+\alpha,$ then
  if $\min_{v\in\Bbb R^{n-k}} h_2(v)=0$
 at $v^*,$  $(0,v^*)$ is an optimal solution of \eqref{dang2} and the optimal value is then zero. 
 However, this situation does not occur when  $\lambda_{\min}(\bar{H},\bar{M})=\lambda_{\min}(W_3,M_3).$
 \begin{lem}
   If  $\lambda_{\min}(\bar{H},\bar{M})=\lambda_{\min}(W_3,M_3)$ then 
   $\min_{v\in\Bbb R^s} h_2(v)>0,$ and thus $\alpha$ must satisfy $\alpha>\sum_{i=1}^s\dfrac{q_i^2}{\gamma_i}.$
 \end{lem}
 \proof  Since $\bar{H}\succeq0$ and $\bar{M}\succeq0,$ we have 
 \begin{align*}
 0\le\lambda_{\min}(\bar{H},\bar{M}) =&\min_{(v,t)^T\in\Bbb R^{s+1}}\dfrac{v^TW_3v-2q^Tvt+\alpha t^2}{v^TM_3v+t^2}\\
 \le & \min_{v\in\Bbb R^s}\dfrac{v^TW_3v-2q^Tv +\alpha }{v^TM_3v+1} =
 \min_{v\in\Bbb R^s}\dfrac{h_2(v)}{v^TM_3v+1}\le \min_{v\in\Bbb R^s} h_2(v).
 \end{align*}
 If $\min_{v\in\Bbb R^s} h_2(v)=0$ then  $\lambda_{\min}(\bar{H},\bar{M})=0.$ However, since  $W_3\succ0$ and $M_3\succ0,$ 
 we have  $\lambda_{\min}(W_3,M_3)>0.$ This shows $\lambda_{\min}(\bar{H},\bar{M})<\lambda_{\min}(W_3,M_3)$ and we get a contradiction.
 So $\min_{v\in\Bbb R^s} h_2(v)>0.$ Let $v^*$ be the minimizer of $h_2(v),$ 
 we have $v^*=W_3^{-1}q$ and $h_2(v^*)=-q^TW_3^{-1}q+\alpha>0.$
 This requires $\alpha>\sum_{i=1}^s\dfrac{q_i^2}{\gamma_i}.$
 \endproof 
 
 Since  $W_3\succ0$ and $M_3\succ0,$  they are SDC \cite[Lemma 6]{Bong}, i.e.,  there exists a nonsingular matrix
$N$ such that  
$$N^TW_3N={\rm diag}(\eta_1, \eta_2, \ldots, \eta_s), ~ N^TM_3N={\rm diag}(\xi_1, \xi_2, \ldots, \xi_s).$$
Then $\lambda_{\min}(W_3,M_3)$ is easy to determine as follows:
$$\lambda_{\min}(W_3,M_3)=\lambda_{\min}(N^TW_3N,N^TM_3N)=\min_{1\le i\le s}\dfrac{\eta_i}{\xi_i}:=\dfrac{\eta_r}{\xi_r}.$$
Let $v=Nz$ and  $L={\rm diag}(I_k, N).$ We make a change of variables  $\left(\begin{matrix} u\\ v\end{matrix}\right)=L\left(\begin{matrix} u\\ z\end{matrix}\right).$
Then
\begin{align}\label{SubSDC}
  \mathcal{H}(u,Nz)&=\dfrac{u^TA_1u-2p^Tu+z^T(N^TW_3N)z-2(N^Tq)^Tz+\alpha}
  {u^TM_1u+2(N^TM_2^Tu)^Tz+z^T(N^TM_3N)z+1}+u^TD_1u\nonumber\\
  &=\dfrac{u^TA_1u-2p^Tu+\sum_{i=1}^s\eta_iz_i^2-2\sum_{i=1}^s\bar{q}_iz_i+\alpha}
  {u^TM_1u+2(N^TM_2^Tu)^Tz+\sum_{i=1}^s\xi_iz_i^2+1}+u^TD_1u,
\end{align}
where $N^Tq:=\bar{q}=(\bar{q}_1, \bar{q}_1, \ldots,\bar{q}_s)^T.$
\begin{lem}\label{B2}
  If $\lambda_{\min}(\bar{H},\bar{M})=\lambda_{\min}(W_3,M_3)$ then 
  \begin{align*} 
    \min_{v\in\Bbb R^s}\dfrac{v^TW_3v-2q^Tv+\alpha}{v^TM_3v+1}=\lambda_{\min}(\bar{H},\bar{M})
  \end{align*}
  and the optimal value $\min_{v\in\Bbb R^s}\dfrac{v^TW_3v-2q^Tv+\alpha}{v^TM_3v+1}$  is not attained.
  \end{lem}
  \proof 
  Let $\nu^*=\min_{v\in\Bbb R^s}\dfrac{v^TW_3v-2q^Tv+\alpha}{v^TM_3v+1}.$ We have
  \begin{align*} 
  \lambda_{\min}(\bar{H},\bar{M})   & = \min_{0\ne(v,t)^T\in\Bbb R^{s+1}}\dfrac{v^TW_3v-2tq^Tv+\alpha t^2}{v^TM_3v+t^2} \\
     & \le  \min_{v\in\Bbb R^s}\dfrac{v^TW_3v-2q^Tv+\alpha}{v^TM_3v+1} \\
      & = \min_{z\in\Bbb R^s}\dfrac{\sum_{i=1}^s \eta_iz_i^2-2\sum_{i=1}^s\bar{q}_iz_i +\alpha}{\sum_{i=1}^s \xi_iz_i^2  +1}\\
    &=\min_{z\in\Bbb R^s}\dfrac{\sum_{r\ne i=1}^s \eta_i\frac{z_i^2}{z_r^2}+\eta_r-2\sum_{r\ne i=1}^s\bar{q}_i\frac{z_i}{z_r^2}
+\frac{\alpha}{z_r^2}}{\sum_{r\ne i=1}^s \xi_i\frac{z_i^2}{z_r^2}+\xi_r  +\frac{1}{z_r^2}}
 \xrightarrow{z_r\rightarrow\infty} \dfrac{\eta_r}{\xi_r}\\
 & =\lambda_{\min}(W_3,M_3)
 =\lambda_{\min}(\bar{H},\bar{M}).
  \end{align*}
  We observe that $\nu^*=\dfrac{\eta_r}{\xi_r}$ but is not attained.
  \endproof

 Let 
 $$\hat{H}=\left(\begin{matrix}
                 A_1 &0 & -p \\
                 0& W_3 & -q \\
                 -p^T & -q^T & \alpha 
               \end{matrix}\right), \hat{M}=\left(\begin{matrix}
                 M_1 &M_2 & 0 \\
                 M_2^T& M_3 & 0 \\
                 0 & 0 & 1 
               \end{matrix}\right).$$ 
We then obtain the following important result.
 \begin{thm}\label{main}
Suppose  $\lambda_{\min}(\bar{H},\bar{M})=\lambda_{\min}(W_3,M_3).$ If \eqref{dang2} is attained at
 $u^*\oplus v^*,$ then
 $u^*$ must be nonzero and
 $$\lambda_{\min}(\hat{H},\hat{M})< \lambda^*<\lambda_{\min}(\bar{H},\bar{M}).$$ 
 
 \end{thm}
\proof 
Suppose that \eqref{dang2} is attained
at $u^*\oplus v^*:$
\begin{align}\label{BDB}
\lambda^*=\underbrace{\dfrac{{u^*}^TA_1u^*-2p^Tu^*+{v^*}^TW_3v^*-2q^Tv^*+\alpha}{{u^*}^TM_1u^*+2(M_2^Tu^*)^Tv^*+{v^*}^TM_3v^*+1}}
_{f(u^*,v^*)}+{u^*}^TD_1u^*.
\end{align}
If $u^*=0,$ \eqref{BDB} is then
\begin{align}\label{BDB1}
\lambda^*=\dfrac{{v^*}^TW_3v^*-2q^Tv^*+\alpha}{{v^*}^TM_3v^*+1}
\ge \lambda_{\min}(\bar{H},\bar{M})=\lambda_{\min}(W_3,M_3)=\dfrac{\eta_r}{\xi_r}.
\end{align}
As seen, \eqref{BDB1} can be written as
\begin{align*} 
\lambda^*=&\dfrac{\sum_{i=1}^s \eta_iz_i^2-2\sum_{i=1}^sq_iz_i
+\alpha}{\sum_{i=1}^s \xi_iz_i^2  
+1}\\
=&\dfrac{\sum_{r\ne i=1}^s \eta_i\frac{z_i^2}{z_r^2}+\eta_r-2\sum_{r\ne i=1}^sq_i\frac{z_i}{z_r^2}
+\frac{\alpha}{z_r^2}}{\sum_{r\ne i=1}^s \xi_i\frac{z_i^2}{z_r^2}+\xi_r  
+\frac{1}{z_r^2}}
\ge \dfrac{\eta_r}{\xi_r}
\end{align*}
 Let $z_r\rightarrow\infty$ we will have $\lambda^*\rightarrow \dfrac{\eta_r}{\xi_r}$ but not attained.
 This contradiction shows that  $u^*\ne0.$ 
 We suppose that $ \lambda^*\ge\lambda_{\min}(\bar{H},\bar{M}).$ Then
\begin{align*} 
\lambda_{\min}(\bar{H},\bar{M})&\le \lambda^*=\dfrac{{u^*}^TA_1u^*-2p^Tu^*+{v^*}^TW_3v^*-2q^Tv^*+\alpha}{{u^*}^TM_1u^*+2(M_2^Tu^*)^Tv^*+{v^*}^TM_3v^*+1}
+{u^*}^TD_1u^*\\
&< \min_{u=0,v\in\Bbb R^s} \dfrac{v^TW_3v-2q^Tv+\alpha}{v^TM_3v+1}=\lambda_{\min}(\bar{H},\bar{M}),
\end{align*}
where the last inequality is due to  $u^*\ne0,$ and the last equality is due to Lemma \ref{B2}.  
This is again a contradiction.
 Finally, $\lambda_{\min}(\hat{H},\hat{M})< \lambda^*$ holds
naturally since 
\begin{align*}\lambda_{\min}(\hat{H},\hat{M})=&\min_{0\ne(u,v,t)^T\in\Bbb R^{k+s+1}}
\dfrac{u^TA_1u-2tp^Tu+v^TW_3v-2tq^Tv+\alpha t^2}{u^TM_1u+2t(M_2^Tu)^Tv+v^TM_3v+t^2}\\
\le&\min_{0\ne(u,v,1)^T\in\Bbb R^{k+s+1}}
\dfrac{u^TA_1u-2p^Tu+v^TW_3v-2q^Tv+\alpha }{u^TM_1u+2(M_2^Tu)^Tv+v^TM_3v+1}\\
< & \min_{0\ne(u,v)^T\in\Bbb R^{k+s}}\left\{
\dfrac{u^TA_1u-2p^Tu+v^TW_3v-2q^Tv+\alpha }{u^TM_1u+2(M_2^Tu)^Tv+v^TM_3v+1} +u^TD_1u\right\}=\lambda^*.
\end{align*}
The last inequality is strict because $u^*\ne0.$
\endproof

Theorem \ref{main} immediately gives the following consequence. 
\begin{cor}
  If    
 $\lambda_{\min}(\hat{H},\hat{M})=\lambda_{\min}(\bar{H},\bar{M})=\lambda_{\min}(W_3,M_3)$ 
 then \eqref{dang2} is not attained. 
\end{cor}
We consider the following simple example to see this conclusion. 
\begin{exam}
 $$ \lambda^*=\inf_{(u,v)^T\in\Bbb R^2}\left\{\dfrac{u^2+v^2+2}{u^2+v^2+1}+u^2\right\}.$$
\end{exam}
Here we have $W_3=(1)=M_3, \bar{H}=\left(\begin{matrix}1&0\\0&2 \end{matrix}\right), 
\bar{M}=\left(\begin{matrix}1&0\\0&1 \end{matrix}\right), \hat{H}=\left(\begin{matrix}1&0&0\\0&1&0\\0&0&2 \end{matrix}\right),
\hat{M}=\left(\begin{matrix}1&0&0\\0&1&0\\0&0&1 \end{matrix}\right).$ Then
$\lambda_{\min}(\hat{H},\hat{M})=\lambda_{\min}(\bar{H},\bar{M})=\lambda_{\min}(W_3,M_3).$ 
The problem is obviously not attained. Indeed, we have
$\mathcal{H}(u,v)\ge 1~\forall (u,v)^T\in\Bbb R^2.$  Let $u=0$ and $v\rightarrow\infty,$ then $\mathcal{H}(u,v)\rightarrow 1$
but $\lambda^*=1$ is not attained.

We now return to (TRTLS). As mentioned, (TRTLS) was formulated to address the problem of solving the overdetermined linear system 
$Ax\approx b,$ with $A\in\Bbb R^{m\times n}, b\in\Bbb R^m.$ We observe that if
$A$ is a square nonsingular matrix, the system $Ax\approx b$ then has a unique solution $x=-A^{-1}b$
and we do not need to use the TLS method. Thus, the main focus on solving 
(TRTLS) is when ${\rm rank}(A)<n.$ 
We rewrite (TRTLS) in the following form
\begin{align*}
{\rm(TRTLS)} \hspace*{0.7cm}\lambda^*:= \min_{x\in\Bbb R^n}\left\{f(x)=\dfrac{x^TA^TAx-2(A^Tb)^Tx+\|b\|^2}{\|x\|^2+1}+\rho x^TL^TLx\right\}.
\end{align*}
Let $A^TA=W, \rho L^TL=B, A^Tb=a$ and $\|b\|^2=\alpha,$ (TRTLS) then becomes the problem (P) and thus can always be transformed 
to one of the three forms, either \eqref{dang1} or \eqref{dang3} or \eqref{dang2}. Especially, 
if ${\rm rank}(A)<n,$  $A^TA$ is then singular so is $W.$ In this case,  $W$ can only be 
transformed to either  \eqref{a} or \eqref{c} while $B$ is always put to \eqref{a0}. 
As a result, (TRTLS) is then only reformulated to either
\eqref{dang1} or \eqref{dang3}.
\begin{enumerate}
  \item  If (TRTLS) is decomposed into \eqref{dang1},   
      then, by Theorem \ref{Thm1}, it is attained if and only if $b=0.$ The optimal solutions are then $0\oplus v,~v\in\Bbb R^s.$
  \item  If (TRTLS)  is decomposed into \eqref{dang3}, i.e., $P^TA^Tb=p\oplus q\oplus0,$ where $p=(p_1, \ldots, p_k)^T\in\Bbb R^k, q=(q_1, \ldots, q_s)^T\in\Bbb R^{s}, 0\in\Bbb R^{n-k-s},$ then by Theorem \ref{Thm2}, it is attained at 
  $u^*\oplus v^*\oplus \omega^*\in\Bbb R^k\times\Bbb R^s\times R^{n-k-s}$
  if and only if  $u^*=0,$ 
\begin{align}\label{cmHq2}
\sum_{i=1}^s\gamma_i{v_i^*}^2-2\sum_{i=1}^kq_i v_i^*+\|b\|^2=0
\end{align}
and $\omega^*\in\Bbb R^{n-k-s}$ is arbitrary. 
Since  
$$h(u,v):=\sum_{i=1}^k\beta_iu_i^2-2\sum_{i=1}^kp_iu_i +\sum_{i=1}^s\gamma_iv_i^2-2\sum_{i=1}^kq_iv_i+\|b\|^2\ge0 ~\forall u\oplus v\in\Bbb R^{k+s},$$
 we have
$h(0,v)\ge h(u^*,v^*)=0~ \forall v\in\Bbb R^s.$
This implies 
 \begin{align}\label{cmHq2_1}
 \sum_{i=1}^s\gamma_iv_i^2-2\sum_{i=1}^kq_iv_i+\|b\|^2\ge0~\forall v \in\Bbb R^s.
 \end{align}
 Both \eqref{cmHq2} and \eqref{cmHq2_1}   indicate that $v^*$ is a minimizer of the convex quadratic  function 
 $$g(v)=\sum_{i=1}^s\gamma_iv_i^2-2\sum_{i=1}^kq_iv_i+\|b\|^2.$$ 
 It is thus the root of the equation 
 $\nabla g(v)=0.$ Solving the equation we have $v_i^*=\dfrac{q_i}{\gamma_i}, i=1,2,\ldots, s,$
 and substituting those into \eqref{cmHq2} we obtain 
   $\|b\|^2=\sum_{i=1}^s\dfrac{q_i^2}{\gamma_i}.$  
   The optimal solutions of (TRTLS) are then of the form $0\oplus v^*\oplus \omega,$
   where $v_i^*=\dfrac{q_i}{\gamma_i}, i=1,2,\ldots, s,$ and $\omega\in\Bbb R^{n-k-s}$ is arbitrary.
   As a result, the attainment and optimal solutions of (TRTLS) in this case are as follows.
 \begin{cor}
   If (TRTLS) is transformed to \eqref{dang3} then it is attained if and only if  
   $\|b\|^2=\sum_{i=1}^s\dfrac{q_i^2}{\gamma_i}.$ The optimal solutions are then  
   $0\oplus v^*\oplus \omega^*\in\Bbb R^k\times\Bbb R^s\times R^{n-k-s}$ such that
    $v_i^*=\dfrac{q_i}{\gamma_i}, i=1,2,\ldots,s,$ and  $\omega^*\in\Bbb R^{n-k-s}$ is arbitrary.  
 \end{cor}
\end{enumerate}

\section{Conclusion}
We have proposed a method of applying the simultaneous diagonalization via congruence of  matrices to solve the
Tikhonov regularization of the total least squares problem (TRTLS) with singular regularized penalty.  
We show that there is always a transformation to reformulate the problem to one of three simpler formats such that each of which
is then clearly characterized. Especially, if ${\rm rank}(A)<n,$ (TRTLS) is then completely solved: 
its attainment is completely characterized and its optimal solution set is found. 
The case ${\rm rank}(A)=n$ is more complicated. 
However, in that case the over-estimated linear system $Ax\approx b,$ which leads to  (TRTLS) has a unique solution.


\begin{thebibliography}{10}
		


		
	\bibitem{Beck06}
	Beck A.,  A. Ben-Tal (2006), On the Solution of the Tikhonov Regularization of the Total Least Squares Problem.  SIAM J. OPTIM. Vol. 17, No. 1, pp. 98-118
	
\bibitem{Bjorck} A. Björck (1996), Numerical Methods for Least Squares Problems, SIAM, Philadelphia
%
%
%
%
%
%


\bibitem{Dinkelbach} W. Dinkelbach: On nonlinear fractional programming. Manag. Sci. 13(7), 492-498 (1967)


\bibitem{Jiang}   R. Jiang, L. Duan,   Simultaneous diagonalization of matrices and its applications in quadratically constrained quadratic programming,  SIAM J. Optim. Vol. 26, No. 3  1649-1668 (2016)


\bibitem{Gould} N.I.M. Gould, D.P. Robinson, and H. Sue Thorne, On solving trust-region and other regularised subproblems in
optimization, Math. Program. Comput. 2 (2010), pp. 21-57.
	
%
	
%
	
	\bibitem{More93}
	Mor\'{e}, J.J., 1993. {\em Generalizations of the trust region problem}. Optimization methods and Software, 2(3-4), pp.  189-209.

\bibitem{Ngan} T. N. Nguyen, V. B. Nguyen, T. H. Le, R. L. Sheu, On Simultaneous Diagonalization via Congruence of
Real Symmetric Matrices and Applications, https://arxiv.org/abs/2004.06360v5	

\bibitem{Bong} V. B. Nguyen, T. N. Nguyen (2024), Positive semidefinite interval of matrix pencil and
its applications to the generalized trust region subproblems, Linear Algebra Appl., 680, 371-390

\bibitem{Bong1} V. B. Nguyen, Solving a generalization of the p-regularized subproblems by
	
	\bibitem{More-Sorensen83}
  Mor\'{e} J J,  Sorensen DC (1983). {\em Computing a trust region step}. SIAM Journal on Scientific and Statistical Computing, 4(3), 553-572.
  
\bibitem{Gander}  W. Gander, G. H. Golub, and U. von Matt, A constrained eigenvalue problem, Linear
Algebra Appl., 114/115 (1989), pp. 815-839

 \bibitem{Gene} Gene H. Golub, P. Christian Hansen,   Dianne P. \'{O'}Leary, Tikhonov Regularization and Total Least Squares,
SIAM J. MATRIX ANAL. APPL., Vol. 21, No. 1, pp. 185-194

 \bibitem{Golub80} Golub, G.H., Van Loan, C.F.: An analysis of the total least-squares problem. SIAM J. Numer. Anal.

 \bibitem{Golub96} Golub, G.H., Van Loan, C.F.: Matrix Computations, 3rd edn. The Johns Hopkins University Press,
Baltimore (1996)

	
%
%
%
%
%
%
\bibitem{Van} Van Huffel, S., Vandewalle, J.: The Total Least Squares Problem: Computational Aspects and Analysis.
Frontiers in Applied Mathematics. SIAM, Philadelphia (1991)


\bibitem{Sorensen} D. C. Sorensen, 1982, Newton’s method with a model trust-region modification, SIAM J. Numer. Anal., 19,
pp. 409-426.
	
	\bibitem{Tikhonov0} 	A. N. Tikhonov,   V. Y. Arsenin, Solutions of Ill-Posed Problems, Wiley, New York, 1977.
	
	
	\bibitem{Tikhonov1} A. N. Tikhonov, “ Regularization of Incorrectly Posed Problems,” Soviet Mathematics Doklady, Vol. 4, No. 6, 1963, pp. 1624-1627

\bibitem{Yang} M. Yang,  Y. Xia,  J. Wang, J. Peng, Efficiently solving total least squares with Tikhonov
identical regularization, Comput Optim Appl, Vol. 70,  571-592, (2018)
	\end{thebibliography}
\end{document}